\documentclass[a4paper,6pt]{article}

\usepackage{a4wide}
\usepackage{authblk}
\usepackage{cite}
\usepackage[british]{babel}
\usepackage[utf8]{inputenc}
\usepackage{dsfont}
\usepackage{graphicx,subfigure}
\usepackage{epstopdf}
    \graphicspath{{./}{figures/}}
\usepackage[colorlinks=true,linkcolor=blue,citecolor=blue]{hyperref}
\usepackage{amsfonts,amsmath,amsthm}

\newcommand{\R}{\mathbb R}

\def\be#1\ee{\begin{equation}#1\end{equation}}
\newcommand{\fer}[1]{(\ref{#1})}
\newtheorem{theorem}{Theorem}[section]

\newcommand{\bq}{\begin{equation}}
\newcommand{\eq}{\end{equation}}
\newenvironment{equations}{\equation\aligned}{\endaligned\endequation}
\def\bqa{\begin{eqnarray}}
\def\eqa{\end{eqnarray}}

\def\ab{{\alpha,\beta}}
\def\ab1{{\alpha_1,\beta_1}}

\newcommand{\bd}{\begin{displaymath}}
\newcommand{\ed}{\end{displaymath}}
\newcommand{\ba}{\begin{eqnarray}}
\newcommand{\ea}{\end{eqnarray}}

\def\R{\mathbb{R}}

\newcommand{\rev}{\textcolor{black}}

\theoremstyle{plain}

\allowdisplaybreaks

\begin{document}
%
%

\title{Trends to equilibrium for a nonlocal Fokker-Planck equation}

\author[1,3]{F. Auricchio \thanks{\texttt{ferdinando.auricchio@unipv.it}}} 
\author[2,3]{G. Toscani \thanks{\texttt{giuseppe.toscani@unipv.it}}}
\author[2]{M. Zanella \thanks{\texttt{mattia.zanella@unipv.it}}}

\affil[1]{Dept. of Civil Engineering and Architecture, University of Pavia} 
\affil[2]{Dept. of Mathematics ``F. Casorati'', University of Pavia} 
\affil[3]{IMATI "E. Magenes", CNR, Pavia}

\date{}

\maketitle

\begin{abstract}
We obtain equilibration rates for a one-dimensional nonlocal Fokker–Planck equation with time-dependent diffusion coefficient and drift, modeling the relaxation of a large swarm of robots, feeling each other in terms of their distance, towards the steady profile characterized by uniform spreading over  a finite interval of the line. The result follows by combining  entropy methods  for quantifying the decay of the solution towards its quasi-stationary distribution, with the  properties of the quasi-stationary  profile.
\end{abstract}
\medskip



\section{Introduction}

The main aim of this work is to study the large-time behavior of the probability density function $f(x,t)$, solution of the one-dimensional nonlocal Fokker-Planck equation
\begin{equation}
\label{eq:FP0_k}
\partial_t f(x,t) = \partial_x \left[ (x-\tilde x_0(t)) f(x,t) + \partial_x (\kappa(x,t) f(x,t)) \right],
\end{equation}
where the initial probability density function $f_0(x) \in L^1(\R)\cap L^\infty(\R)$ is such that 
\[
\int_\R(1 + x^2 + \log f_0(x))\,f_0(x) \, dx < +\infty.
\]
In equation \fer{eq:FP0_k}, for a given $x_0 \in \R$, and $\lambda,\mu>0$ such that $\lambda +\mu =1$, 
\be\label{x_0}
\tilde x_0(t) = \lambda x_0 + \mu u(t),
\ee
and, for any time $t \ge 0$, the mean value of $f(x,t)$ is given by 
\be\label{mean}
u(t) = \int_\R xf(x,t)\, dx.
\ee
In \fer{eq:FP0_k}  the continuous variable diffusion coefficient $\kappa(x,t)$, 
\rev{uniformly bounded from below and above}, is given by
\begin{equation}
\label{eq:k0}
\kappa(x,t) = \kappa(x-\tilde x_0(t)) =
\begin{cases}
\sigma^2 + \dfrac{\delta^2}{2}-\dfrac{1}{2}|x-\tilde x_0(t)|^2 & |x-\tilde x_0(t)|<\delta \\
\sigma^2 & |x-\tilde x_0(t)|\ge \delta.
\end{cases}
\end{equation}
 Equation \fer{eq:FP0_k}  is intended to describe the action of a large swarm of moving agents,  capable to spread uniformly over the surface of a \rev{target domain $D = [x_0-\delta,x_0+\delta] \subset \R$, see   \cite{ATZ}}, while the presence of the mean value $u(t)$ is related to a suitable communication between them \rev{and the coefficient $\sigma^2>0$ represents possible disturbances in the motion}. This simple task can be interpreted as the deposition of a single layer in standard additive manufacturing processes \cite{Ac,D_etal,HW,Ox}, a topic included in the analysis of self-organizing features of mathematical models for large systems in social and life sciences \cite{AP,DOB,During,K_etal,MT},  often described  by kinetic-type equations  \cite{BCC,CFRT,HT,PT}. The problem to determine asymptotic behaviour of the system is a classical task \cite{Carrillo,Alonso,Gamba}. 

Concerning the Fokker--Planck equation \fer{eq:FP0_k} , the simpler case $\mu =0$ $\lambda =1$, corresponding to absence of communication between particles, has been treated in \cite{ATZ}, where it was noticed that  the target action of the swarm can be achieved by different choices of the Fokker--Planck type equations, corresponding to alternative strategies.  If $\mu = 0$ in equation \fer{eq:FP0_k}, particles move subject to the simultaneous presence of the drift and diffusion operators, and, still under the action of the drift operator, they  start to randomly explore the target domain $D$ adapting their diffusion to the distance from the center $x_0$ of the domain,  as soon as  they enters in it. In this case, it has been shown in \cite{ATZ} that the solution to the resulting Fokker-Planck type equation \fer{eq:FP0_k} converges in time towards the steady profile with a polynomial rate. 
 
In this paper we  extend the analysis of \cite{ATZ} to the case in which  particles of the swarm are not equally informed about the target to reach, but they sense the entire swarm in terms of their relative distance showing that the solution to the Fokker-Planck equation \rev{still} converges in time to the right target distribution \rev{but at a lower polynomial rate}.

\section{Fokker-Planck models in swarm manufacturing}

We consider a system of $N\gg 1$ particles interacting between them and  with a \rev{target  domain} $D \subset \mathbb R$, which for simplicity is assumed to be an interval centered in $x_0$ with lenght $2\delta>0$. 

Each particle senses the direction of motion towards the center of  $D$ and, once in $D$, it starts to randomly explore the target domain. At variance with \cite{ATZ}   each particle modifies its position also through interactions with the other particles of the swarm.  

Let $f(x,t)\,dx$ be the probability of finding a particle in the elementary volume $dx$ around the point $x \in \mathbb R$ at time $t \ge 0$. The evolution in time of the density $f(x,t)$ can be furnished  by a Fokker-Planck-type equation with constant diffusion and discontinuous and time-dependent drift
\begin{equation}
\label{eq:FP}
\partial_t f(x,t) = \partial_x  \left[ \mathcal B[f](x,t)\rev{\mathds{1}_{D^c}{(x)}}f(x,t) + \sigma^2 \partial_x f(x,t) \right].
\end{equation}
In \fer{eq:FP} the constant $\sigma^2>0$ characterizes the speed of diffusion of the particles, and 
\begin{equation}
\label{eq:drift1}
\mathcal B[f](x,t) = \lambda (x-x_0) + \mu \int_{\mathbb R} P(x, y)(x -y)f( y,t)d y,
\end{equation}
quantifies the attractiveness of the domain $D$ in presence of interactions among  particles. 
Finally, \rev{we indicated with $D^c = \mathbb R\setminus D$ the complement of $D$ and \rev{$\mathds{1}_A(x)$} the indicator function of the domain $A \subseteq \R$}.

In \fer{eq:drift1} $\lambda,\mu$ are nonnegative constants such that $\lambda + \mu = 1$ and $P(x,\cdot) \ge 0$ is an interaction function weighting the influence on a particle in \rev{ $x$ }of all the other particles in terms of their distance from $x$. 
 Therefore,  particles move subject to the simultaneous presence of drift and diffusion unless they are in the target domain $D$ where only the diffusion operator survives. 
 
In what follows we consider $P\equiv 1$, namely the simplified situation characterized by a uniform interaction rate. In this case, the drift term \eqref{eq:drift1} is expressed in terms of the mean value \fer{mean}
\begin{equation}
\label{eq:drift}
\begin{split}
\mathcal B[f](x,t) &= x-\lambda x_0 - \mu u(t) =x - \tilde x_0(t) .
\end{split}
\end{equation}
Hence, the presence of a uniform interaction between particles can be translated in mathematical terms by saying that  the center $\tilde x_0 (t)$ of the interval of attraction of lenght $2\delta$, as defined in \eqref{x_0},   is now moving with time.
To avoid the presence of a discontinuous drift, in \cite{ATZ}  we noticed that the same target dynamics (the same steady profile) could be obtained by resorting to the Fokker--Planck type equation \fer{eq:FP0_k}, characterized by a continuous drift and a variable continuous diffusion $\kappa(x,t)$, as given by \fer{eq:k0}.

To clarify why it is more convenient  to resorting to   formulation  \fer{eq:FP0_k}, it is enough to compute the evolution of the mean value $u(t)$ of the solution to both Fokker--Planck equations. 
As \rev{far} as the formulation \fer{eq:FP} is concerned,  one easily obtains 
\[
\begin{split}
\dfrac{d\, u(t) }{dt} &= - \int_{\R} \mathcal B[f](x,t)\rev{\mathds{1}_{D^c}(x)} f(x,t)dx  \\
&= - \lambda \int_{\rev{D^c}} (x-x_0)f(x,t)dx - \mu \int_{\rev{D^c}}\int_{ \R}(x-y)f(y,t)f(x,t)d y dx. 
\end{split}\]
It is immediate to conclude that, even if  $\mu =0$ the evolution of the mean value  is not explicitly computable
in reason of the presence of the  discontinuous drift.  On the contrary, the mean value of the solution to equation \fer{eq:FP0_k} solves
\[
\dfrac{d\,  u(t) }{dt}= -  \rev{\int_{\R} (x-\tilde{x}_0(t))f(x,t)dx}
= - \lambda( u(t)- x_0), 
\]
and
\be\label{meant}
  u(t) -x_0 = ( u(0) -x_0) e^{-\lambda\, t}.
\ee
\rev{Furthermore, if initially bounded, the second order moment of the solution to equation \fer{eq:FP0_k} remains  bounded. Indeed, if $\bar \kappa = \max_{x \in \mathbb R,t\ge 0}\kappa(x,t)$, 
\[
\begin{split}
\frac{d}{dt} \int_\R x^2\, f(x,t)\,dx&= 2 \int_{\mathbb R}\kappa(x,t)f(x,t)dx -2 \int_{\mathbb R}(x^2-\tilde x_0 x)f(x,t)dx \\
& \le 2 \bar{\kappa} + 2 u(t)\tilde{x}_0(t)-2\int_\R x^2\, f(x,t)\,dx.
 \end{split}
\]
} Since  $u(t)$ converges exponentially to $x_0$ at a rate $\lambda$ for $t\to +\infty$, we can expect that the solution to the Fokker--Planck equation  \eqref{eq:FP0_k} will converge for large time to the steady state  obtained in \cite{ATZ} in the absence of communications between particles, given by
\begin{equation}
\label{eq:finfty_1}
f_\infty(x) = 
\begin{cases}
\dfrac{m_1}{\sqrt{2\pi \sigma^2}} \exp\left\{-\dfrac{|x-x_0|^2}{2\sigma^2}\right\} & |x-x_0|\ge \delta,\\
\dfrac{m_2}{2\delta} & |x-x_0|< \delta, 
\end{cases}
\end{equation}
where $m_1,m_2>0$ only depend on the the conditions that the total mass is unitary and the steady state is continuous over the domain. As shown in \cite{ATZ},  the masses $m_1,m_2$ are uniquely determined in terms of the relevant parameters $\delta,\sigma^2$ characterizing respectively the diffusion coefficient and the length of the interval. Hence, the  steady state \fer{eq:finfty_t} is a continuous function resulting from the weighted combination of a Gaussian density outside $D$ and a uniform density inside $D$. 

To study convergence of the solution to the Fokker--Planck equation \fer{eq:FP0_k} towards the steady state \fer{eq:finfty_1} we remark that, in view of its structure, this equation possesses a \emph{quasi-stationary} solution, namely a solution, for a fixed time $t>0$, of the first-order differential equation 
\[
 \partial_x (\kappa(x -\tilde x_0(t)) f(x,t)) +  (x-\tilde x_0(t)) f(x,t) =0.
 \]
This solution, in the time-independent case $\mu = 0$, coincides with \fer{eq:finfty_1}, whereas, for $\mu>0 $, it is given by 
\begin{equation}
\label{eq:finfty_t}
f_q(x,t) = 
\begin{cases}
\dfrac{m_1}{\sqrt{2\pi \sigma^2}} \exp\left\{-\dfrac{|x-\tilde x_0(t)|^2}{2\sigma^2}\right\} & |x-\tilde x_0(t)|\ge \delta,\\
\dfrac{m_2}{2\delta} & |x-\tilde x_0(t)|< \delta \rev{.}
\end{cases}
\end{equation}


\section{Large-time behaviour} 

Convergence to equilibrium will be mainly based on the study of the time decay of entropy functionals \cite{FPTT,T1,TZ}, and it will be reached in three steps. We remark that all the computations that follow are justified by the properties of the diffusion coefficient $\kappa(x,t)$ of the Fokker-Planck equation \fer{eq:FP0_k}, that allow us to apply Proposition $2$ of Section $6$ of the paper by Le Bris and Lions \cite{LLB}, relative to the existence and uniqueness of solutions to Fokker-Planck type equations with irregular time dependent coefficients of diffusion and drift.  We have

\begin{theorem}\label{main}
Let us consider the initial value problem for the Fokker--Planck equation  \fer{eq:FP0_k} .
Then, for each \rev{probability density} $f_0(x) \in L^1(\R) \cap L^\infty(\R)$, equation  \fer{eq:FP0_k}  has a unique  solution in the space 
\begin{equations}\nonumber
&f(x,t) \in  L^\infty([0,T], L^1 \cap L^\infty),\\ 
&\rev{\kappa(x,t)} \partial_x f(x,t)\in L^2([0,T], L^2).
\end{equations}
\rev{The unique solution $f(x,t)$ is a probability density for any subsequent time $t>0$}. If moreover the initial datum $f_0(x)$ satisfies
\begin{equation}\label{eq:f0_int}
\int_{\R}(1 +|x|^2 + \log f_0(x))\,f_0(x) \, dx < +\infty \rev{,}
\end{equation}
then, for all \rev{$t \in [0,T]$ }
\be\label{HH}
\int_{\R}(1 +|x|^2 + \log f(x,t))\,f(x,t) \, dx < +\infty.
\ee
\end{theorem}

\begin{proof}
\rev{The proof of existence and uniqueness follows from a result by Le Bris and Lions \cite{LLB}, since the drift and diffusion coefficients of equation \fer{eq:FP0_k} satisfy the hypotheses of  Proposition $2$ of Section $6$. Indeed  $x-\tilde{x}_0(t) \in L^1([0,T], W^{1,1}_{\textrm{loc}}(\mathbb R))$ and it has constant derivative. Moreover
\[
\frac{x-\tilde{x}_0(t)}{1+|x|} = \dfrac{1}{(1+|x|)^2} + \dfrac{(x-\tilde x_0(t))(1+|x|)-1}{(1+|x|)^2} \in L^1([0,T], L^1 + L^\infty(\mathbb R)). 
\]
Also $\kappa(x,t) \in L^2([0,T], W^{1,2}_{\textrm{loc}}(\mathbb R))$ and $\dfrac{\kappa(x,t)}{1+|x|} \in L^2([0,T], L^2+L^\infty(\mathbb R))$. 
} 

\rev{The positivity of the solution  may be obtained by rewriting equation to \eqref{eq:FP0_k} in an equivalent way. To this extent, let us define $g(y,t) = f(x,t)$, where, for $x \in \R$, $y= x-\tilde{x}_0(t)$. then $g(y,t)$ solves
\begin{equation}
\label{eq:g}
\partial_t g(y,t) = - K(t)\partial_y g(y,t) + \partial_y \left[ yg(y,t) + \partial_y(\kappa(y)g(y,t)) \right] = \mathcal A g(y,t) + \mathcal B g(y,t).
\end{equation}
In \fer{eq:g} $K(t) = \lambda\mu(u(0)-x_0)e^{-\lambda t}$. The equation is composed of a pure transport operator $\mathcal A$ and a drift-diffusion operator $\mathcal B$ with $\kappa(y)\ge\sigma^2>0$. Both operators are linear and positivity preserving. Applying the splitting method to equation \fer{eq:g},  and Trotter's formula \cite{DM}
\[
e^{t(\mathcal A + \mathcal B)} = \lim_{n\to +\infty} \left( e^{\mathcal A t/n }e^{\mathcal B t/n}\right)^n
\]
allows to conclude with the positivity of $g(y,t)$, and, consequently of $f(x,t)$.
}
\rev{Concerning the validity of \eqref{HH},  since the solution to \eqref{eq:FP0_k} has bounded second order moment, we may argue as in \cite{Arkeryd} to conclude that $ f |\log^-f|\le e^{-|x|^2} + |x|^2 f  \in L^1(\mathbb R)$. Hence, being $f \log^+ f \in L^1(\mathbb R)$,  $f \log f$ belongs to $L^1(\mathbb R)$.  }
\end{proof}

\rev{We now investigate the asymptotic behaviour of the solution}.  The following result holds
 
 \begin{theorem}\label{convergence}
Let  $f(x,t)$ be the unique solution to the initial value problem for the one-dimensional Fokker--Planck equation \fer{eq:FP0_k}, departing from an initial condition $f_0(x)$ satisfying the hypotheses of Theorem \ref{main}. Then  $f(x,t)$ converges in $L^1(\R)$  towards the one-dimensional steady solution $f_\infty(x)$ defined by \fer{eq:finfty_1}  and the convergence rate is at least $o(t^{-1/2})$.
\end{theorem}
\begin{proof}
Let $H(g|h)$ be the \rev{relative} Shannon entropy between two probability density functions $g$ and $h$  \be\label{eq:entr}
 H(g|h) = \int_\R g(x) \log \dfrac{g(x)}{h(x)}\, dx.
 \ee
Since $f$ and $f_q$ are time-dependent we have
 \be\label{eq:decay}
 \dfrac{d}{dt}  H(f(t)|f_q(t)) = \int_\R \left(1 +\log\dfrac{f(x,t)}{f_q(x,t)}  \right)\partial_t f(x,t) \, dx  - \int_\R f(x,t)\partial_t \log f_q(x,t)  \, dx.
 \ee
\rev{Thanks to Theorem \ref{main} we obtain that $f(x,t) \in H^1(\mathbb R)$ and therefore is sufficiently regular.} The first integral on the right-hand side of equation \fer{eq:decay} coincides with $-I_H(f(t)|f_q(t))$,  where following \rev{Theorem 7 in} \cite{FPTT}, the entropy production term of the relative Shannon entropy is given by 
\be\label{ep}
I_H(f(t)|f_q(t)) = \rev{4}\int_{\R} \kappa(x,t)f_q(x,t) \left(\partial_x \sqrt{\dfrac{f(x,t)}{f_q(x,t)}} \right)^2\, dx .
\ee
To evaluate the second term, we observe that, resorting to definition \fer{eq:finfty_t} one obtains
\begin{equation}
\label{eq:finfty_der}
\partial_t f_q(x,t) = 
\begin{cases}
-f_q(x,t) \dfrac{\lambda\mu}{\sigma^2}(u(0) -x_0)(x-\tilde x_0(t))e^{-\lambda\, t}  & |x-\tilde x_0(t)|> \delta,\\
0 & |x-\tilde x_0(t)|< \delta, 
\end{cases}
\end{equation}
Hence, the second term in \eqref{eq:decay} can be evaluated by resorting to \fer{eq:finfty_der} 
\[
\begin{split}
&L_H(f(t)|f_q(t)) = - \int_\R f(x,t)\partial_t \log f_q(x,t)  \, dx = e^{-\lambda\, t} \dfrac{\lambda\mu}{\sigma^2}(u(0) -x_0)\int_{ |x-\tilde x_0(t)|> \delta}  (x-\tilde x_0(t))f(x,t)\, dx  = \\ 
&\quad e^{-\lambda\, t} \dfrac{\lambda\mu}{\sigma^2}(u(0) -x_0)\left[ \rev{\lambda}(u(0) -x_0)  e^{-\lambda\, t}  - \int_{ |x-\tilde x_0(t)|< \delta}  (x-\tilde x_0(t))f(x,t)\, dx \right] .
\end{split}\]
Furthermore, we have
\[
\left|\int_{ |x-\tilde x_0(t)|< \delta}  (x-\tilde x_0(t))f(x,t)\, dx \right| \le \int_{ |x-\tilde x_0(t)|< \delta}  |x-\tilde x_0(t)|f(x,t)\, dx  \le \delta,
\]
and the term $L_H(f(t)|f_q(t))$ decays exponentially to zero at the rate $\lambda t$.

As a second step, let us observe that the entropy production term \fer{ep} bounds from above the square of the Hellinger distance between $f(x,t)$ and $f_q(x,t)$ \cite{ATZ}, where the Hellinger distance between two probability density functions  $g,h$ is defined as follows
\[
D^2(g|h) = \int_{\mathbb R} \left(\sqrt{g(x)}-\sqrt{h(x)}\right)^2\,dx. 
\]
Indeed, as shown for example in \rev{Section 3.2 of} \cite{FPTT} one has the inequality
\be\label{disu}
D^2(f(t)| f_q(t)) \le \rev{\dfrac{1}{2} }I_H(f(t)|f_q(t)).
\ee
As a third step, let us evaluate the variation in time of the Hellinger distance. We have  \cite{FPTT} 
 \be\label{eq:decayD}
 \dfrac{d}{dt}D^2(f(t)|f_q(t)) = \int_\R \left(1 - \sqrt{\dfrac{f_q(x,t)}{f(x,t)} } \right)\partial_t f(x,t) \, dx + \int_\R \left(1 - \sqrt{\dfrac{f(x,t)}{f_q(x,t)} } \right)\partial_t f_q(x,t) \, dx .
 \ee
 The first integral on the right-hand side of equation \fer{eq:decayD} coincides with $-I_D(f(t)|f_q(t))$,   where the entropy production term of the squared Hellinger distance is given by  \cite{FPTT}
\be\label{epH}
I_D(f(t)|f_q(t)) =\rev{8} \int_{\R} \kappa(x,t)f_q(x,t) \left(\partial_x \sqrt[4]{\dfrac{f(x,t)}{f_q(x,t)}} \right)^2\, dx .
\ee
The second term on the right-hand side of \eqref{eq:decayD} corresponds to 
\begin{equations}\label{eq2}
&L_D(f(t)|f_q(t)) = \\
-&e^{-\lambda\, t} \dfrac{\lambda\mu}{\sigma^2}(u(0) -x_0)\int_{ |x-\tilde x_0(t)|> \delta}  (x-\tilde x_0(t))\left(f_q(x,t)- \sqrt{f(x,t)f_q(x,t)}\right)\, dx.
\end{equations}
\rev{ Substituting $f_q(x,t)$ on the set $ |x-\tilde x_0(t)| \ge \delta$  with a Gaussian density, it is immediate to show,  by the Cauchy-Schwarz inequality and  definition \fer{eq:finfty_t}, that }
 \begin{equations}\nonumber
&\left| \int_{ |x-\tilde x_0(t)|> \delta}  (x-\tilde x_0(t))f_q(x,t)\, dx\right| \le  \left[\int_{ |x-\tilde x_0(t)|> \delta}  (x-\tilde x_0(t))^2f_q(x,t)\, dx\right]^{1/2} \le \\
&\left(\int_{\R}  (x-\tilde x_0(t))^2 \dfrac{m_1}{\sqrt{2\pi \sigma^2}} \exp\left\{-\dfrac{|x-\tilde x_0(t)|^2}{2\sigma^2}\right\}\, dx\right)^{1/2}   = \frac{\sqrt{m_1}}{\sigma}.
 \end{equations} 
 The same inequality holds for the second integral in \fer{eq2}. Therefore we get
\be
|L_D(f(t)|f_q(t))| \le 2 \frac{\sqrt{m_1}}{\sigma}e^{-\lambda\, t} \dfrac{\lambda\mu}{\sigma^2}|u(0) -x_0| = -2 \frac{\sqrt{m_1}\mu}{\sigma^3} |u(0) -x_0|\frac d{dt}e^{-\lambda\, t}. 
\ee
The previous bound, combined with  \fer{eq:decayD} implies 
\be\label{giu}
\dfrac{d }{dt}D^2(f(t)|f_q(t)) + 2 \frac{\sqrt{m_1}\mu}{\sigma^3} |u(0) -x_0|\frac d{dt}e^{-\lambda\, t} \le -I_D(f(t)|f_q(t)) \le 0.
\ee
 To conclude, it is enough to remark that, integrating the equation \fer{eq:decay} in time from $0$ to $\infty$ shows that the entropy production $I_H(f(t)|f_q(t))$ is integrable over the positive real line. Hence, thanks to inequality \fer{disu}, we realize that $D^2(f(t)|f_q(t))$ is integrable in time on the positive real line. On the other hand, thanks to \fer{giu}, we know that the function
 \[
 D^2(f(t)|f_q(t)) + 2 \frac{\sqrt{m_1}\mu}{\sigma^3} |u(0) -x_0|e^{-\lambda\, t}
 \]
is non-increasing and integrable in time. Consequently, the square of the Hellinger distance decays to zero as time goes to infinity at least at a rate $o(1/t)$. This shows that the Hellinger distance, and consequently the $L^1(\R)$-norm, between the solution $f(x,t)$ of the Fokker--Planck equation \fer{eq:FP0_k} and its quasi-steady solution decays to zero as time goes to infinity at a rate at least of order $o(t^{-1/2})$. It remains to prove that  $f_q(x,t)$ converges in $L^1(\R)$-norm, as time tends to infinity, towards the equilibrium solution $f_\infty(x)$, as given by \fer{eq:finfty_1}, to conclude with the convergence in $L^1(\R)$-norm of the solution to the  Fokker--Planck equation \fer{eq:FP0_k} towards $f_\infty(x)$.

Rather than directly demonstrating the convergence in $L^1(\R)$-norm of $f_q(x,t)$  towards $f_\infty(x)$, we evaluate the  Shannon entropy of $f_\infty(x)$ relative to $f_q(x,t)$, that is
\[
 H(f_\infty|f_q(t)) = \int_\R f_\infty(x) \log \dfrac{f_\infty(x)}{f_q(x,t)}\, dx\rev{.}
\]
 Owing to definition \fer{eq:finfty_t}, to the definition of $\tilde x_0(t)$ given in \fer{x_0},  and to \fer{meant}, we realize that, for $t>0$ large enough, $\tilde x_0(t)$ is sufficiently close to $x_0$, and the above integral can be evaluated by splitting it in four domains. We consider the case in which $u(0) -x_0 > 0$, the other case can be treated likewise, and let us set $B(t) = (u(0) -x_0)e^{-\lambda t}$. The domains are defined by
 \begin{equations}\label{set4}
 &E_1(t) = \left\{ x < x_0 -\delta; x > x_0+ \delta + B(t)\right\}; E_2(t) =  \left\{  x_0 -\delta +B(t) < x < x_0+ \delta \right\}; \\
 &E_3(t) = \left\{ x_0 -\delta <x < x_0- \delta + B(t)\right\}; E_4(t) =  \left\{ x_0 +\delta <x < x_0+\delta + B(t)\right\}.
  \end{equations}
  On the domain $E_1(t)$ the functions $f_\infty(x)$ and $f_q(x,t)$ coincide with the exponential functions. Consequently
  \[
  \int_{E_1(t)} f_\infty(x) \log \dfrac{f_\infty(x)}{f_q(x,t)}\, dx =(\tilde x_0 - x) \int_{E_1(t)} f_\infty(x)\frac 1{2\sigma^2} ( \tilde x_0(t) + x_0 -2x)\, dx,
  \]
which converges to zero exponentially in time since all moments of the exponential distribution $f_\infty$ are bounded. 

On the set $E_2(t)$ the the functions $f_\infty(x)$ and $f_q(x,t)$ coincide with the same constant and the integral is therefore equal to zero.  Finally, let us remark that the sets $E_3(t)$ and $E_4(t)$ have measure $B(t)$, that decays exponentially to zero. 
Since the functions inside the integrals are bounded in absolute value, by the mean value theorem we conclude with the exponential decay to zero of the integrals evaluated on these two sets. 

In conclusion, the relative entropy $ H(f_\infty|f_q(t))$ decays exponentially to zero. We remark that an explicit bound can be provided, at the expense of additional computations.

On the other hand, convergence in relative entropy implies, thanks to the Csiszar-Kullback-Leibner inequality \cite{FPTT}, convergence in $L^1(\R)$. Indeed, for any pair of probability density functions  $g$ and $h$ with the same mass, this inequality reads
\[ 
\|g-h\|_{L^1(\R)}^2 \le 2 H(g|h).
\]
This concludes the proof of convergence. 
\end{proof}

\section*{Concluding remarks}

Starting from a Fokker-Planck-type model with discontinuous drift  describing large swarms of agents that spread uniformly over the surface of a domain, we have introduced a communication between agents that can sense each other and align towards the center of the target domain.  The introduced modeling approach has been inspired by printing tasks and mimics the deposition of a single layer in additive manufacturing processes, and the goal is to exploit the cooperative nature of the interacting systems to define robust and distributed strategies. 

As shown in \cite{ATZ}, the right target can be achieved by resorting to a Fokker--Planck type equations like \fer{eq:FP0_k}, which shares the key property to possess a continuous drift. For this equation, existence and uniqueness of the solution is known \cite{LLB}. Here, we  studied the large time behavior of the solution,  obtaining an explicit rate of convergence in the prototypical case of unitary communication strength. The results extend the ones developed in \cite{ATZ} to the case of variable diffusion coefficient. Future works will extend the results to variable  communication functions.

\medskip

{\bf Acknowledgement.} This work has been done under the activities of the
National Group of Mathematical Physics (GNFM). GT wish to acknowledge partial support by IMATI, Institute for Applied Mathematics and Information Technologies “Enrico Magenes”, Via Ferrata 5 Pavia, Italy. MZ acknowledges the support of MUR-PRIN2020 Project No.2020JLWP23 (Integrated Mathematical Approaches to Socio–Epidemiological Dynamics).


\medskip

\begin{thebibliography}{10}

\bibitem{Ac}
E. Ackerman. Mobile Robots Cooperate to 3D Print Large Structures. \emph{IEEE Spectrum: Technology, Engineering, and Science News}, 28 August 2018.

\bibitem{AP}
G. Albi, L. Pareschi. Modeling of self-organized systems interacting with a few individuals: From microscopic to macroscopic dynamics. \emph{Appl. Math. Lett.}, \textbf{4}:397--401, 2013. 

\bibitem{ATZ}
F. Auricchio, G. Toscani, M. Zanella. Fokker-Planck modeling of many-agent systems in swarm manufacturing: asymptotic analysis and numerical results. \emph{Commun. Math. Sci.}, in press. 

\bibitem{Alonso}
 R. Alonso, V. Bagland, L. Desvillettes, B. Lods. About the use of entropy production for the Landau-Fermi-Dirac equation. \emph{J. Stat. Phys.}, \textbf{183}(1): Paper n.10, 2021. 
 
 \bibitem{Arkeryd}
 L. Arkeryd. On the Boltzmann equation. Part I: Existence. \emph{Arch. Ration. Mech. Anal.}, \textbf{45}: 1--16, 1972. 

 \bibitem{Gamba}
N. Ben Abdallah, I. Gamba, G. Toscani. On the minimization problem of sub-linear convex functionals. \emph{Kinet. Relat. Models}, \textbf{4}(4):857--871, 2011. 

\bibitem{BCC}
F. Bolley, J. Cañizo, J. A. Carrillo. Stochastic mean-field limit: non-Lipschitz forces and swarming. \emph{Math. Mod. Meth. Appl. Sci.}, \textbf{21}:2179--2210, 2011. 

\bibitem{CFRT}
J.A. Carrillo, M. Fornasier, J. Rosado, G. Toscani. Asymptotic flocking dynamics for the kinetic Cucker–Smale model. \emph{SIAM J. Math. Anal.}, \textbf{42}:218--236, 2010. 

\bibitem{Carrillo}
J. A. Carrillo, S. Hittmeir, B. Volzone, T. Yao, Nonlinear aggregation-diffusion equations: radial symmetry and long time asymptotics. \emph{Invent. Math.}, \textbf{218}(3):889–977, 2019.

\bibitem{DM}
L. Desvillettes, M. Mischler. About the splitting algorithm for Boltzmann and B.G.K. equations. \emph{Math. Models Methods Appl. Scie.} \textbf{6} (8) 1079--1101, 1996.

\bibitem{DOB}
M.R. D’Orsogna, Y.L. Chuang, A.L. Bertozzi, L.S. Chayes. Self-propelled particles with soft-core interactions: patterns, stability, and collapse. \emph{Phys. Rev. Lett.}, \textbf{96}:104-302, 2006. 

\bibitem{D_etal}
S. Duncan, G. Estrada-Rodriguez, J. Stocek, M. Dragone, P. Vargas,  H. Gimperlein. Efficient quantitative assessment of robot swarms: coverage and targeting Levy strategies. \emph{Bioinspir. Biomim.}, \textbf{17}(3), 2022.

\bibitem{During}
B. Düring, M. Torregrossa, M.-T. Wolfram. Boltzmann and Fokker-Planck equations modelling the Elo rating system with learning effects. \emph{J. Nonlinear Sci.}, \textbf{29}(3):1095–1128, 2019. 

\bibitem{FPTT}
G. Furioli, A. Pulvirenti, E. Terraneo, G. Toscani. Fokker-Planck equations in the modeling of socio-economic phenomena. \emph{Math. Models Methods Appl. Scie.}, \textbf{27}(1):115-158, 2017. 

\bibitem{HT}
S.-Y. Ha, E. Tadmor. From particle to kinetic and hydrodynamic descriptions of flocking. \emph{Kinet. Relat. Models}, \textbf{1}:415--435, 2008.

\bibitem{HW}
H. Hamann, H. W{\"o}rn. A framework of space–time continuous models for algorithm design in swarm robotics. \emph{Swarm Intelligence}, \textbf{2}(2): 209--239, 2008.

\bibitem{K_etal}
A. J. King, S. J. Portugal, D. Strömbom, R. P. Mann, J. A. Carrillo, D. Kalise, G. de Croon, H. Barnett, P. Scerri, R. Gro\ss, D. R. Chadwick, M. Papadopoulou. Biologically inspired herding of animal groups by robots. \emph{Math. Ecol. Evol}, 00, 1--9, 2023. 

\bibitem{LLB}
C. Le Bris, P.-L. Lions. Existence and uniqueness of solutions to Fokker-Planck type equations with irregular coefficients. \emph{Commun. Partial Differ. Equ.}, \textbf{33}(7):1272--1317, 2008. 

\bibitem{MT}
S. Motsch, E. Tadmor. Heterophilious dynamics enhances consensus. \emph{SIAM Rev.}, \textbf{56}(4):577--621, 2014. 

\bibitem{Ox}
N. Oxman, J. Duro--Royo, S. Keating, B. Peters, E. Tsai. Towards robotic swarm printing. \emph{Architectural Design}, \textbf{84}(3):108--115, 2014.

\bibitem{PT}
L. Pareschi, G. Toscani. \emph{Interacting Multiagent Systems: Kinetic Equations \& Monte Carlo Methods}, Oxford University Press, 2013. 

\bibitem{T1}
G. Toscani. Entropy dissipation and the rate of convergence to equilibrium for the Fokker-Planck equation. \emph{Quart. Appl. Math.}, \textbf{LVII}:521--541, 1999. 

\bibitem{TZ}
G. Toscani, M. Zanella. On a class of Fokker-Planck equations with subcritical confinement. \emph{Rend. Lincei Mat. Appl.}, \textbf{32}:471--497,2021. 

\end{thebibliography}
\end{document}